\documentclass[12pt]{amsart} 
\usepackage{fullpage}
\usepackage{amsmath,amssymb}

\usepackage{graphicx} 
\usepackage{graphicx} 
\usepackage{amssymb} 
\usepackage{epstopdf} 

\usepackage[all,cmtip]{xy} 

\theoremstyle{plain} 
\newtheorem{thm}{Theorem} 
\newtheorem{cor}[thm]{Corollary} 
\newtheorem{lem}[thm]{Lemma}

\theoremstyle{definition} 
\newtheorem{defn}[thm]{Definition} 
\newtheorem{rem}[thm]{Remark} 
\newtheorem{ex}[thm]{Example} 
\newtheorem{exs}[thm]{Examples} 

\newcommand{\iinfty}{{\mathchoice
{\begin{minipage}{.15in}\includegraphics[width=.15in]{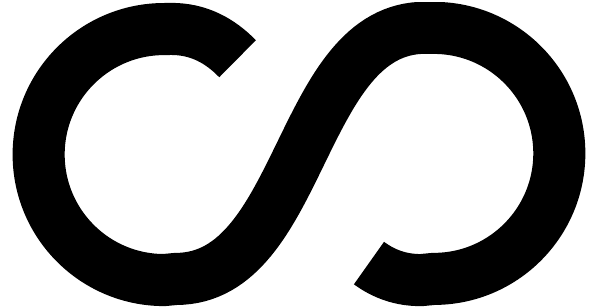}\end{minipage}}
{\begin{minipage}{.12in}\includegraphics[width=.12in]{infty2.pdf}\end{minipage}}
{\begin{minipage}{.10in}\includegraphics[width=.10in]{infty2.pdf}\end{minipage}}
{\begin{minipage}{.08in}\includegraphics[width=.08in]{infty2.pdf}\end{minipage}}
}}

\renewcommand{\int}{\operatorname{int}}

\newcommand{\id}{\operatorname{id}}

\newcommand{\im}{\operatorname{Im}}

\newcommand\sL{\text{\sf L}}

\newcommand{\Z}{\mathbb{Z}} 
 
\newcommand{\R}{\mathbb{R}}

\newcommand{\bL}{\mathbb{L}} 
\newcommand{\bT}{\mathbb{T}}

\newcommand{\g}{\mathfrak{g}}

\newcommand{\cL}{\mathcal{L}} 
 
\newcommand{\cT}{\mathcal{T}} 
 
\newcommand{\cW}{\mathcal{W}}

\newcommand{\sra}{\twoheadrightarrow} 
 
\newcommand{\ra}{\longrightarrow}

\newcommand{\QF}{\text{\sf{QF}}}
\newcommand{\CQF}{\text{\sf{CQF}}}
\newcommand{\SQF}{\text{\sf{SQF}}}
\newcommand{\SF}{\text{\sf{SF}}}
\newcommand{\HF}{\text{\sf{HF}}}
\newcommand{\QR}{\text{\sf{QR}}}

\newcommand{\sg}{\operatorname{signature}}

\newcommand{\z}{\mathbb Z_2}

\begin{document} 

\title{Universal quadratic forms and Whitney tower intersection invariants} 

\dedicatory{Dedicated to Mike Freedman on the occasion of his 60th birthday.}

\begin{abstract} 
We first remind the reader of a simple geometric description of the Kirby-Siebenmann invariant of a 4--manifold in terms of a quadratic refinement of its intersection form. This is the first in a sequence of higher-order intersection invariants of Whitney towers, studied by the authors, particularly for the 4--ball. 

In the second part of this paper, a general theory of quadratic forms is developed and then specialized from the non-commutative to the commutative to finally, the symmetric settings. The intersection invariant for twisted Whitney towers is shown to be the universal symmetric refinement of the framed intersection invariant. As a corollary we obtain a short exact sequence that has been essential in the understanding of Whitney towers in the 4--ball.
\end{abstract}

\author[J. Conant]{James Conant} 
\email{jconant@math.utk.edu} 
\address{Dept. of Mathematics, University of Tennessee, Knoxville, TN} 

\author[R. Schneiderman]{Rob Schneiderman} 
\email{robert.schneiderman@lehman.cuny.edu} 
\address{Dept. of Mathematics and Computer Science, Lehman College, City University of New York, Bronx, NY} 

\author[P. Teichner]{Peter Teichner} 
\email{teichner@mac.com} 
\address{Dept. of Mathematics, University of California, Berkeley, CA and} 
\address{Max-Planck Institut f\"ur Mathematik, Bonn, Germany}


\maketitle 
\section{Introduction}
A beautiful consequence of Mike Freedman's disk embedding theorem is the existence of non-smoothable $4$-manifolds. In the easiest setting, his result can be stated as follows.

\begin{thm}\label{thm:classification}
Any odd unimodular symmetric form $\lambda:\Z^m \otimes \Z^m \to \Z$ is realized as the intersection form of exactly {\em two} closed simply-connected oriented 4--manifolds (up to homeomorphism). These 4--manifolds are homotopy equivalent and are distinguished by the following (equivalent) criteria: Exactly one of the manifolds \dots
\begin{enumerate}
\item \dots is smoothable after crossing with $\R$.
\item \dots is smoothable after connected sum with finitely many copies of $S^2 \times S^2$.
\item \dots has a linear reduction of its micro normal bundle.
\item\dots has vanishing Kirby-Siebenmann invariant in $\Z_2$.
\item\dots exhibits the following formula for a quadratic refinement $\tau$ of $\lambda$:
\[
\tau(c) \equiv \frac{\lambda(c,c) - \sg(\lambda)}{8}  \mod 2 \quad \forall \text{ characteristic elements } c.
\] 
\end{enumerate}
\end{thm}
\begin{rem} By Donaldson's Theorem A \cite{D}, {\em exactly} the diagonalizable odd forms $\lambda$ are realized by closed {\em smooth} 4--manifolds. Diagonal forms are realized by connected sums of complex projective planes (with varying orientations); in fact, most such forms are now known to admit infinitely many smooth representatives (all being homeomorphic by the above theorem), 
see e.g.~\cite{FS}.
\end{rem}

Criterion (v) is the most elementary and will be explained in detail in Section~\ref{sec:KS}. The key is the following geometric interpretation for the quadratic refinement $\tau(c)$. In a simply-connected closed 4--manifold $M$, any class in $H_2(M)$ can be represented by a (topologically generic) immersed sphere $S: S^2 \to M$. This means that $S$ looks locally like $\R^2 \times 0 \subset \R^4$, except for finitely many double points around which $S$ looks like $\R^2 \times 0 \cup 0 \times \R^2 \subset \R^4$. One can add more local self-intersection points  to $S$ until their algebraic sum is zero. This implies that one can choose {\em Whitney disks} $W_i$, pairing all these self-intersection points. These are (topologically generic) immersed disks $W_i:D^2 \to M$ whose boundary consists of two arcs, each going between the two intersection points but on different sheets, see Figure~\ref{fig:Whitney-move}. 

\begin{figure}[h]
        \centerline{\includegraphics[scale=.5]{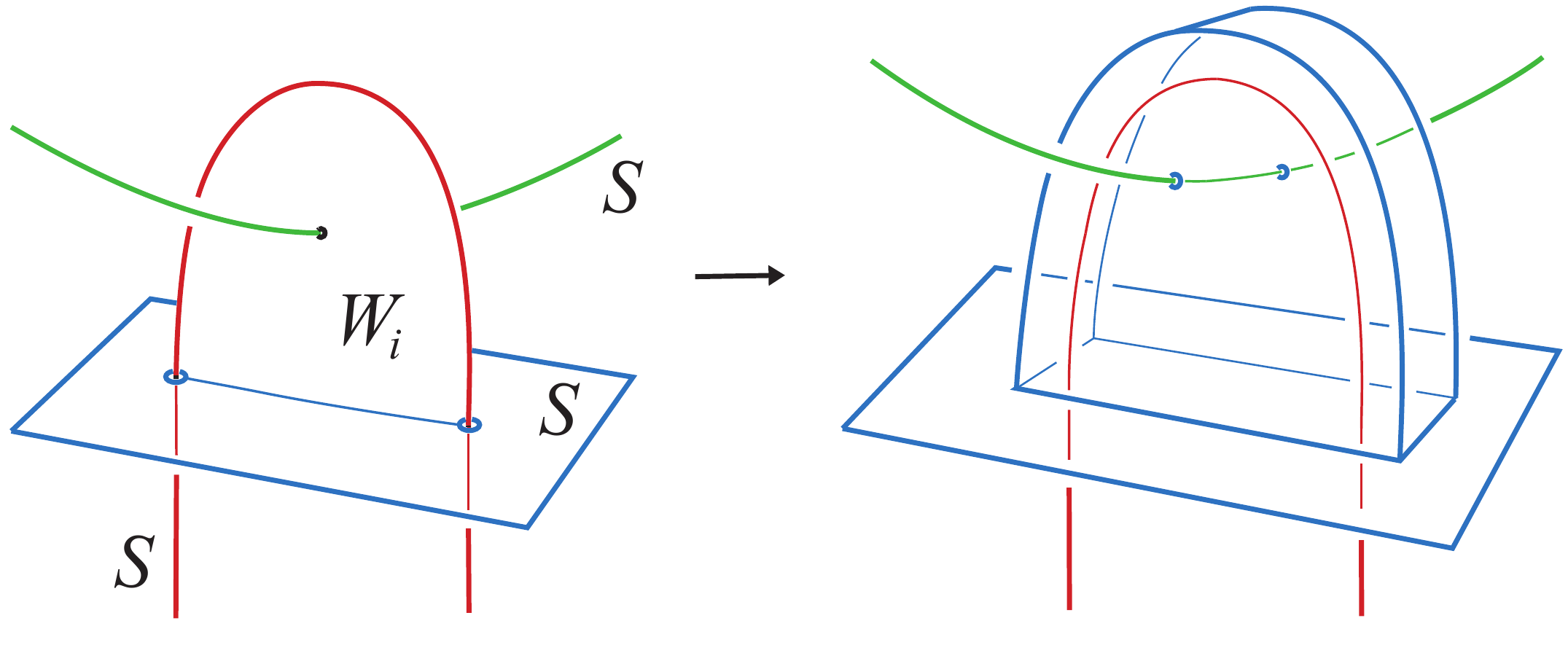}}
        \caption{A (framed) Whitney disk and a Whitney move.}
        \label{fig:Whitney-move}
\end{figure}

We will explain in Lemma~\ref{lem:c} why $\tau(c)$ equals an intersection invariant $\tau_1(S,W_i)$, computed by summing the (topologically generic) intersections between an immersed sphere $S$, representing the characteristic $c\in H_2(M) = \Z^m$, and (the interiors of) framed Whitney disks $W_i$ for $S$:
\[
\tau(c) = \tau_1(S,W_i) := \sum_i \#\{S\pitchfork W_i\}  \quad \mod 2
\]
In \cite{FQ}, this invariant was called the {\em Kervaire-Milnor invariant}  because these authors first proved Rohlin's formula below \cite{KM} for the case where $M$ is smooth and $c$ is represented by an embedded sphere, implying the properties $\tau(c)=0=KS(M)$.  

\begin{rem}\label{rem:framing}
The figure above shows a {\em framed} Whitney disk $W_i$ in the sense that there are two {\em disjoint} parallel copies of $W_i$, as needed for the Whitney move on the right hand side. In general, a Whitney disk comes with a framing of its boundary and hence admits a well defined Euler number in $\Z$, its {\em twist}. The operation of {\em boundary twisting} \cite{FQ} allows to assume that all Whitney disks are framed, i.e.~have twist zero. Moreover, one can also assume that the $W_i$ are (disjointly) embedded disks, by pushing all (self)-intersections off the boundary.
\end{rem}

A generalization of Rohlin's theorem \cite{FK} says that this geometric invariant determines the Kirby-Siebenmann invariant of a closed oriented 4--manifold $M$ by the formula
\[
KS(M) \equiv \tau_M(c) + \frac{\lambda_M(c,c) - \sigma(\lambda_M)}{8}  \mod 2
\]
explaining the equivalence of criteria (iv) and (v) above. In Section~\ref{sec:KS} we'll recall a definition of $\tau_M$ which makes the above formula hold for all closed oriented 4--manifolds $M$ (without assuming that $c$ is spherical).

The 2-complex $\cW:=S\cup W_i$ in $M$ is referred to as a {\em Whitney tower} of order 1, built on $S$, with order 1 Whitney disks $W_i$. The invariant $\tau_1(\cW)=\tau_1(S,W_i)$ used above is the first {\em intersection invariant} of such Whitney towers. It has again {\em order 1}, the order zero intersection invariants being given by the intersection form $\lambda_M$. In a sequence of papers, the authors generalized this invariant to higher orders, see for example our survey \cite{CST0}. 

The idea is that if $\tau_1(\cW)$ vanishes then all intersections between $S$ and $W_i$ can be paired by {\em order 2} Whitney disks $W_{i,j}$ and there should be a second-order intersection invariant $\tau_2(\cW,W_{i,j})$ measuring the obstruction for finding order 3 Whitney disks, and so on.

In \cite{CST1} we worked out this higher-order intersection theory in detail for Whitney towers built on immersed disks in the $4$--ball bounded by framed links in the $3$--sphere. In this simply connected setting the invariant $\tau_n(\cW)$ of an order $n$ (framed) Whitney tower $\cW$ takes values in an abelian group $\cT_k(m)$ (where $m$ is number of link components), and the vanishing of $\tau_n(\cW)$ implies that the link bounds and order $n+1$ Whitney tower. For links bounding \emph{twisted} Whitney towers there is an analogous obstruction theory and intersection invariant
$\tau^\iinfty_k(\cW)\in\cT^\iinfty_k(m)$, and in the second part of this paper we develop an algebraic theory of quadratic forms, leading to a beautiful relation between these framed and twisted obstruction groups, spelled out in Theorem~\ref{thm:exact}. This result is used in the computation of the Whitney tower filtration on classical links described in \cite{CST1}.

{\bf Acknowledgments:} 
The main part of this paper was written while the first two authors were visiting the third author at the Max-Planck-Institut f\"ur Mathematik in Bonn. They all thank MPIM for its stimulating research environment and generous support. The third author was also supported by NSF grants DMS-0806052 and DMS-0757312.

\section{Combinatorial approach to the Kirby-Siebenmann invariant} \label{sec:KS}
In the notation of the introduction, recall that a class $c\in \Z^m$ is called {\em characteristic} if
\[
\lambda(c,x) \equiv \lambda(x,x) \mod 2 \quad \forall x\in \Z^m
\]
The set $C(\lambda)$ of characteristic elements is a $\Z^m$-torsor via the action $(c,x)\mapsto c+2x$. 

A closed oriented 4--manifold $M$ with $(H_2M,\lambda_M) = (\Z^m,\lambda)$ defines a {\em quadratic refinement} $\tau_M:C(\lambda)\to\Z_2$ of $\lambda$ in the sense that 
\[
\tau_M(c+2x) \equiv \frac{\lambda(c,x)-\lambda(x,x)}{2} \mod 2
\]
This means that $\tau_M$ is completely determined by {\em one} of its values (and that value is determined by the Kirby-Siebenmann invariant of $M$ via Rohlin's theorem above). If $\lambda$ is even then $c=0$ is characteristic and a simple geometric argument shows that $\tau_M(0)=0$, see below. That's why $\tau_M$ is only interesting for odd intersection forms. 

Freedman and Kirby \cite{FK,K} defined $\tau_M(c)$ by representing $c$ by an embedded surface $\Sigma \subset M$ such that $M\smallsetminus \Sigma$ has a spin structure that does not extend across $\Sigma$. Let $\pi:S\nu(\Sigma,M) \to \Sigma$
 be the projection map of the boundary of a normal disk bundle for $\Sigma$. Then there is a quadratic refinement $q$ of the intersection form on $H_1(\Sigma)$ defined by taking the inverse image torus $\pi^{-1}(a)$ for a circle $a$ in $\Sigma$ and observing that it comes equipped with a canonical spin structure. Then $q(a) \in \Z_2$ is the Arf invariant of this spin torus, aka its spin bordism class, and finally $\tau_M(c)$ is the Arf invariant of $q$. 
 
This definition can be drastically simplified if $c$ is represented by an immersed sphere $S$ whose self-intersection points are paired by framed Whitney disks $W_1,\dots,W_g$. (This is always the case if $M$ is simply-connected as in the introduction.)
\begin{lem}\label{lem:c}
Define $\tau_1(S,W_i)\in\Z_2$ to be the sum of all intersections between the immersed sphere $S$ and the interiors of the Whitney disks $W_i$. Then $\tau_1(S,W_i) = \tau_M(c)$.
\end{lem}
\begin{proof}
In \cite{FK} the following simplification of $\tau_M(c)$ is already explained, in fact, this was the original definition (and only later it was realized that one needs a more general approach because caps don't always exist):  Assume that $c$ is represented by a surface $\Sigma$ with (immersed, framed) {\em caps}. These are (immersed, framed) disks $A_1,\dots,A_g,B_1,\dots,B_g$ in $M$ bounding a hyperbolic basis $a_1,\dots,a_g,b_1,\dots,b_g$ of curves in $H_1(\Sigma)$. 

Freedman-Kirby show that the quadratic form $q:H_1(\Sigma)\to \Z_2$ is determined by $q(a_i)=$ number of intersections between the interior of the cap $A_i$ and $\Sigma$, similarly for $q(b_i)$.  By definition of the Arf invariant, one gets that
\[
\tau_M(c) = \sum^g_{i=1} q(a_i)\cdot q(b_i)
\]
Assume now that $c$ is represented by an immersed sphere $S$ whose self-intersection points are paired by (immersed, framed) Whitney disks $W_1,\dots,W_k$. We can get into the capped surface situation as follows: For each pair of self-intersection points of $S$, add a tube $T_i$ on one sheet going from one self-intersection to the other. That turns $S$ into an embedded surface $\Sigma$ with half of the caps $A_i$ given by small normal disks to $\Sigma$ that bound the generating circles on $T_i$. Moreover, the Whitney disks $W_i$ can serve as the dual caps $B_i$, preserving the framing, as illustrated in Figure~\ref{fig:framings-preserved}.

\begin{figure}[h]
        \centerline{\includegraphics[scale=.35]{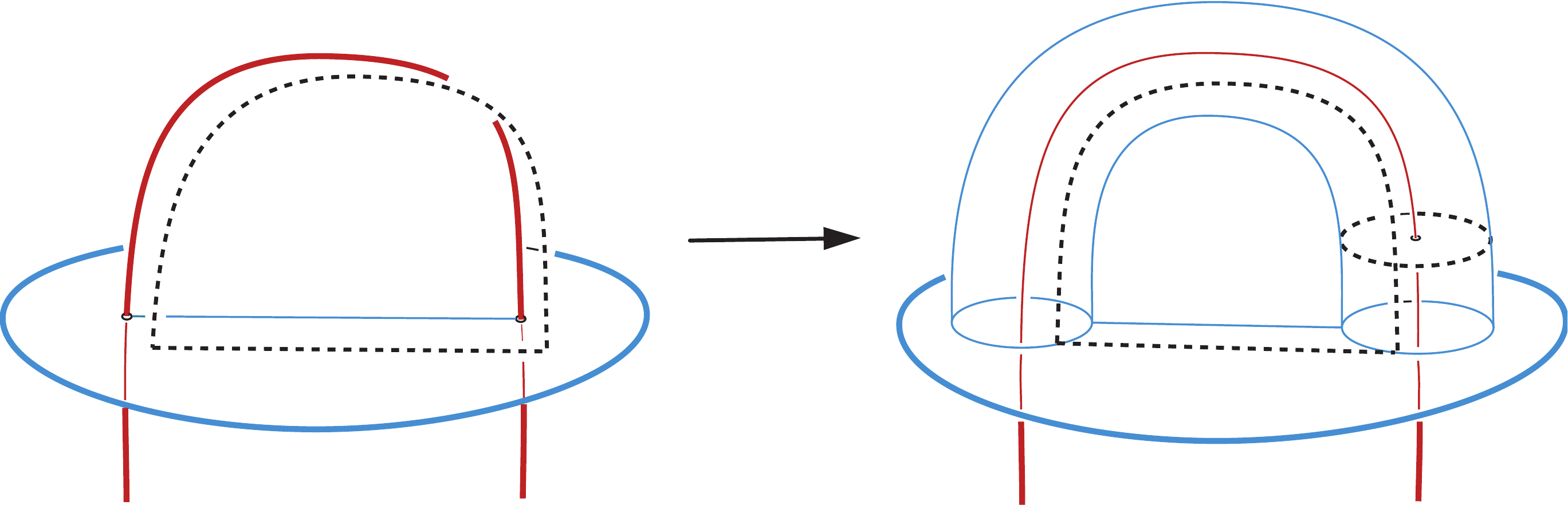}}
        \caption{Turning an immersed sphere with Whitney disks into a capped surface.}
        \label{fig:framings-preserved}
\end{figure}

By construction, $q(a_i)=1$ since each normal disk $A_i$ intersects $\Sigma$ in a single point. Therefore, the required formula follows:
\[
\tau_M(c) = \sum^g_{i=1} q(b_i) = \tau_1(S,W_i)
\]
\end{proof}
\begin{rem}\label{rem:tau well-defined}
In the simply-connected case, it is not hard to see that $\tau_1\in \Z_2$ is well-defined exactly on characteristic elements. One thing to check is that it does not depend on the choices of the Whitney disks $W_i$. Once we fix the boundary, any two such choices differ by a connected sum into a sphere $S_i$. If we require the Whitney disks to be (stably) framed then $S_i$ needs to be (stably) framed and hence it intersects a characteristic sphere in an even number of points, leaving our count $\tau_1$ unchanged modulo two.
\end{rem}
All these considerations can be found in chapter~10 of the book \cite{FQ} by Freedman and Quinn. Unfortunately, the results don't hold as stated for 4--manifolds with fundamental groups that contain 2-torsion elements. The problem arises from different choices of pairings of intersection points, as observed by Richard Stong in \cite{S}. In \cite{ST}, the last two authors gave a complete discussion of the invariant $\tau_1$ in the presence of fundamental groups.

\section{Abelian groups generated by trees}\label{sec:trees}

All trees considered in this paper are \emph{unitrivalent, oriented and labelled}.  This means that they are equipped with \emph{vertex orientations}, i.e.\ cyclic orderings of the edges incident to each trivalent vertex. Moreover, the univalent vertices of a tree are labeled by elements of the index set $\{1,2,\ldots,m\}$, except for the unlabeled {\em root vertex} if a given tree happens to be rooted. (There is at most one root, even though any other index can appear several times). All trees are considered up to label-preserving isomorphism.

The  \emph{order} of such a tree is the number of trivalent vertices. 
\begin{figure}[h]
        \centerline{\includegraphics[scale=.5]{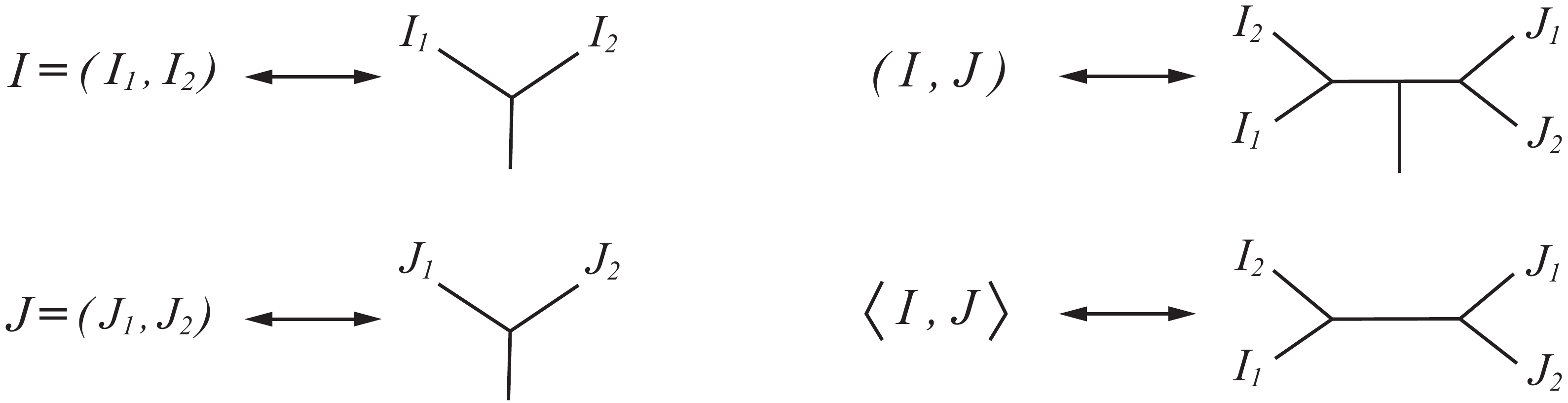}}
        \caption{Rooted and inner products.}
        \label{fig:products}
\end{figure}

Given rooted trees $I$ and $J$, the \emph{rooted product} $(I,J)$ is the rooted tree gotten by identifying the two roots to a vertex and adjoining a rooted edge to this new vertex, with the orientation of the new trivalent vertex given by the ordering of $I$ and $J$ in $(I,J)$. 
The {\em inner product} $\langle I,J \rangle$ of two rooted trees $I$ and $J$ is defined to be the unrooted tree gotten by identifying the two rooted edges to a single edge. 
We observe that the two products interact well in the sense of Figure~\ref{fig:rotate}.
\begin{figure}[h]
        \centerline{\includegraphics[scale=.55]{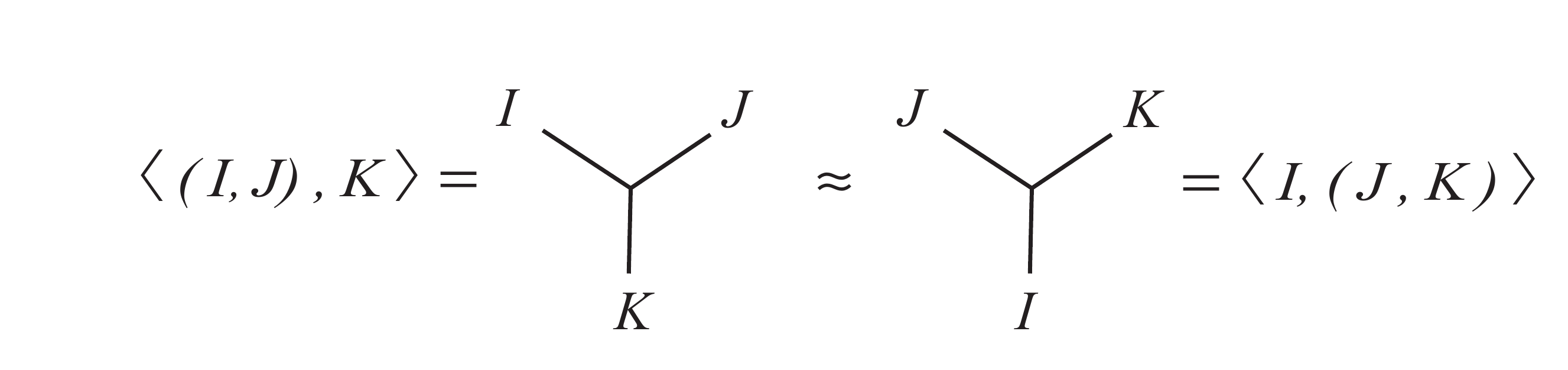}}
        \caption{Invariance of the inner product.}
        \label{fig:rotate}
\end{figure}

Let $\bL(m)$ be the free abelian group generated by (isomorphism classes of) rooted trees as above. It is graded by order and the rooted product can be extended linearly to a pairing
\[
( \ , \ ):\  \bL(m) \otimes \bL(m)\longrightarrow \bL(m)
\]
This is grading preserving on $\bL(m)[1]$, i.e.~it preserves the grading when shifted up by one (so order is replaced by the number of univalent non-root vertices). On the other hand, the inner product 
\[
\langle \ , \ \rangle:\  \bL(m) \otimes \bL(m)\longrightarrow \bT(m)
\]
is grading preserving via order. Here $\bT(m)$ is the free abelian group generated by unrooted trees as above. 

Note that rotating the relevant planar trees by 180 respectively 120 degrees shows that the inner product is both {\em symmetric} and {\em invariant}: $\langle I,J \rangle = \langle J,I \rangle$ and $\langle (I,J),K \rangle = \langle I,(J,K) \rangle$,
see Figure~\ref{fig:rotate} for the proof of invariance.

\begin{defn}\label{def:tree-groups} The graded abelian groups $\cL(m)$ respectively $\cT(m)$ are defined as quotients of $\bL(m)$ respectively $\bT(m)$ by the AS and IHX relations as in Figure~\ref{fig:relations}.
\begin{figure}[h]
\centerline{\includegraphics[scale=.9]{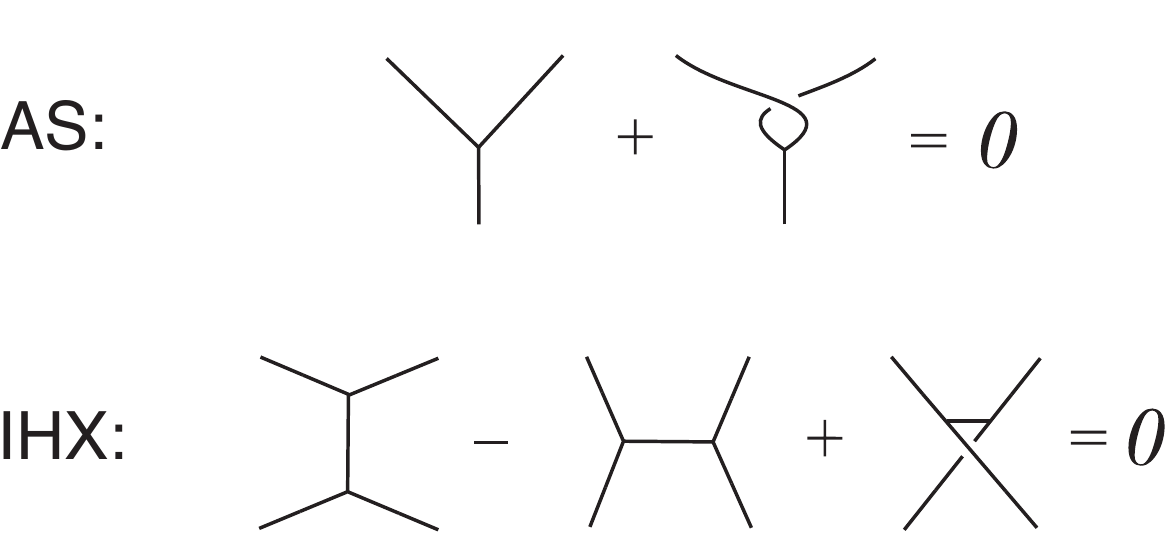}}
         \caption{Local \emph{antisymmetry} (AS) and \emph{Jacobi} (IHX) relations
         in $\cL(m)$ and $\cT(m)$. All trivalent orientations come from an orientation of the plane, and univalent vertices extend to subtrees which are fixed in each equation.}
         \label{fig:relations}
\end{figure}
\end{defn}

It is well known that $\cL(m)$ is the free (quasi) Lie algebra over $\Z$ on $m$ generators with Lie bracket induced by the rooted product. Here the word {\em quasi} refers to the fact that we only require the antisymmetry relations $[X,Y]= - [Y,X]$. So $[X,X]$ is not necessarily zero in these Lie algebras. In our previous papers, we needed to consider both versions of Lie algebras and used the notation $\sL'_{n+1}(m)$ for $\cL_n(m)$ (recall that one gets a {\em graded} Lie algebra only when shifting the order by one). In this paper we will only study one type of Lie algebras and usually omit the adjective `quasi'. 

\begin{rem}
The inner product extends uniquely to a bilinear, symmetric, invariant pairing
\[
\langle \ , \ \rangle:\ \cL(m) \times \cL(m) \longrightarrow \cT(m)
\]
This follows simply from observing that the AS and IHX relations hold on both sides and are preserved by the inner product. We will show in Lemma~\ref{lem:inner product} that this inner product is in fact {\em universal}. 
\end{rem}

\begin{defn}\label{def:twisted-tree-groups} 
The group $\cT^{\iinfty}_{2n}(m)$ is gotten from $\cT_{2n}(m)$ by adding order $n$ $\iinfty$-trees as new generators. These are rooted trees of order $n$ as above, except that the root carries the label $\iinfty$. In addition to the IHX- and AS-relations on unrooted trees in $\cT_{2n}(m)$ , we introduce the 
following new  {\em symmetry}, \emph{interior twist} and \emph{twisted IHX} relations:
\[
J^\iinfty=(-J)^\iinfty\quad\quad
 2\cdot J^\iinfty=\langle J,J\rangle \quad\quad I^\iinfty=H^\iinfty+X^\iinfty- \langle H,X\rangle 
\]
\end{defn}
As their names suggest, these new relations arose from geometric considerations for twisted Whitney towers in \cite{CST1}. They will be explained algebraically in our last section via the theory of universal quadratic refinements. 

Roughly speaking, the universal symmetric pairing $ \langle \ , \ \rangle$ will be shown to admit a universal quadratic refinement $q: \cL_n(m) \to \cT_{2n}^\iinfty(m)$ defined by $q(J):=J^\iinfty$. In particular,  with the right algebraic notion of `quadratic refinement', the group $\cT_{2n}^\iinfty(m)$ is completely determined by the pairing $ \langle \ , \ \rangle$. The rest of this paper is devoted to finding this notion.

As a consequence, we will prove the following exact sequence at the very end of this paper. It was used substantially in  \cite{CST1}  for the classification of Whitney towers in the 4-ball.
\begin{thm}\label{thm:exact}
For all $m,n$, the maps $t\mapsto t$ respectively $J^\iinfty\mapsto 1\otimes J$ give  an exact sequence:
\[
0 \ra \cT_{2n}(m) \ra \cT^\iinfty_{2n}(m) \ra \Z_2 \otimes \cL_{n}(m) \ra 0
\]
\end{thm}


\section{Invariant forms and quadratic refinements}\label{sec:invt-forms-quadratic-refinements}
In this section we explain an algebraic framework into which our groups 
$\cT(m)$ and $\cT^\iinfty(m)$ fit naturally. 
In Lemma~\ref{lem:inner product} we show that the 
$\cT(m)$-valued inner product $\langle\,\cdot\,,\,\cdot\,\rangle$ on the free Lie algebra is universal. Then a general theory of quadratic refinements is developed and specialized from the non-commutative to the commutative to finally, symmetric settings. In Corollary~\ref{cor:quadratic} we show that $\cT_{2n}^\iinfty(m)$ is the home for the universal quadratic refinement of the $\cT_{2n}(m)$-valued inner product $\langle \ , \ \rangle$.

We work over the ground ring of integers but all our arguments go through for any commutative ring. We also only discuss the case of finite generating sets $\{1,\dots,m\}$, even though everything holds in the infinite case.

\subsection{A universal invariant form}

The following lemma shows that this inner product is {\em universal} for Lie algebras with $m$ generators. 
 
\begin{lem} \label{lem:inner product} Let $\g$ be a Lie algebra together with a  bilinear, symmetric, invariant pairing $\lambda:\g \times \g \to M$ into some abelian group $M$. If $\alpha:\cL(m)\to\g$ is a Lie homomorphism (given by $m$ arbitrary elements in $\g$) there exists a {\em unique} linear map $\Psi:\cT(m)\to M$ such that  for all $X,Y\in \cL(m)$
\[
\lambda(\alpha(X), \alpha(Y) ) = \Psi( \langle X,Y \rangle )
\]
\end{lem}
\begin{proof}
The uniqueness of $\Psi$ is clear since the inner product map is onto. For existence, we first construct a map $\psi:\bT(m)\to M$ as follows. Given a tree $t\in \bT(m)$ pick an edge in $t$ to split $t= \langle X,Y \rangle $ for rooted trees $X,Y\in \bL(m)$. Then set
\[
\psi(t) := \lambda(\alpha(X), \alpha(Y))
\]
If we split $t$ at an adjacent edge, this expression stays unchanged because of the symmetry and invariance of $\lambda$. However, one can go from any given edge to any other by a sequence of adjacent edges, showing that $\psi(t)$ does not depend on the choice of splitting.

It is clear that $\psi$ can be extended linearly to the free abelian group on $\bT(m)$ and since $\alpha$ preserves AS and IHX relations by assumption, this extension factors through a map $\Psi$ as required.
\end{proof}

\begin{rem} \label{rem:grading} Recall that $\cL(m)[1]$ is actually a graded Lie algebra, i.e.\ the Lie bracket preserves the grading when shifted up by one (so order is replaced by the number of univalent non-root vertices).
Let's assume in the above lemma that the groups $\g,M$ are $\Z$-graded, $\g[1]$ is a graded Lie algebra and that $\lambda,\alpha$ preserve those gradings. Then the proof shows that the resulting linear map $\Psi$ also preserves the grading.
\end{rem}

\subsection{Non-commutative quadratic groups}
The rest of this section describes a general setting for relating our groups $\cT^\iinfty_{2n}(m)$ that measure the intersection invariant of twisted Whitney towers to a universal (symmetric) quadratic refinement of the $\cT_{2n}(m)$-valued inner product.
We first give a couple of definitions that generalize those introduced by Hans Baues in \cite{B1} and \cite[\S 8]{B2}, and Ranicki in \cite[p.246]{R}. These will lead to the most general notion of quadratic refinements for which we construct a universal example. Later we shall specialize the definitions from {\em non-commutative} to {\em commutative} and finally, to {\em symmetric} quadratic forms and construct universal examples in all cases.
\begin{defn}\label{def:quadratic group} A {\em (non-commutative) quadratic group}
\[
\frak M = (M_{e} \overset{h}{\to} M_{ee} \overset{p}{\to} M_{e})
\]
consists of two groups $M_e, M_{ee}$ and two homomorphisms $h,p$ satisfying 
\begin{enumerate}
\item $M_{ee}$ is abelian,
\item the image of $p$ lies in the center of $M_e$,
\item $hph=2h$. 
\end{enumerate}
$\frak M$ will serve as the range of the (non-commutative) quadratic forms defined below. We will write both groups additively since in most examples $M_e$ turns out to be commutative. A morphism $\beta: \frak M \to \frak{M}'$ between quadratic groups is a pair of homomorphisms 
\[
\beta_e: M_e\to M_e'  \quad \text{ and } \quad \beta_{ee}: M_{ee}\to M_{ee}'
\]
 such that both diagrams involving $h,h',p,p'$ commute.
\end{defn}

\begin{exs} \label{ex:M}
The example motivating the notation comes from homotopy theory, see e.g.\cite{B1}. For $m < 3n-2$, let $M_e= \pi_m(S^n)$, $M_{ee}=\pi_m(S^{2n-1})$, $h$ be the Hopf invariant and $p$ be given by post-composing with the Whitehead product $[\iota_n,\iota_n]:S^{2n-1}\to S^n$.

This quadratic group satisfies $php=2p$ which is part of the definition used in \cite{B1}, where $M_e$ is also assumed to be commutative.  As we shall see, these additional assumptions have the disadvantage that they are not satisfied for the universal example~\ref{ex:universal}. 

Another important example comes from an abelian group with involution $(M,*)$. Then we let
\begin{equation} \tag{$M$}
M_{ee}:=M, \quad M_e:=M/\langle x - x^* \rangle , \quad h([x]):=x+x^*
\end{equation}
and $p$ be the natural quotient map.  For example, if $M$ is a ring with involution $r\mapsto \bar r$, then we get two possible involutions on the abelian group $M$: $r^*= \pm \bar r$. The choice of sign determines whether we study symmetric respectively skew-symmetric pairings.
\end{exs}

We note that in this example $hp-\id=*$ and in the homotopy theoretic example $hp-\id=(-1)^n$. In fact, we have the following
\begin{lem}\label{lem:involutions}
Given a quadratic group, the endomorphism $hp-\id$ gives an involution on $M_{ee}$ (which we will denote by $*$). Moreover, the formula $\dagger(x) := ph(x) -x$ defines an anti-involution on $M_e$. These satisfy
\begin{enumerate}
\item $*\circ h = h$, 
\item $php = p + p\circ *$,
\item $p \circ * = \dagger\circ p$. 
\end{enumerate}
\end{lem}
The proof of Lemma~\ref{lem:involutions} is straightforward and will be left to the reader. To show that $\dagger$ is an anti-homomorphism one uses that $\im(p)$ is central and that $x\mapsto -x$ is an anti-homomorphism. 

\begin{defn}\label{def:quadratic group refinement}
A quadratic group $\frak M$ is a {\em quadratic refinement} of an abelian group with involution $(M,*)$ if
\[
M_{ee}=M  \quad \text{ and } \quad *=hp-\id.
\]
It follows from (i) in Lemma~\ref{lem:involutions} that in this case, the image of $h$ lies in the fixed point set of the involution: $h: M_e \to M^{\z} = H^0(\z;M)$. 
\end{defn}
The example $(M)$ above gives one natural choice of a quadratic refinement, however, there are other canonical (and non-commutative) ones as we shall see in Example~\ref{ex:universal}.

It follows from (ii) in Lemma~\ref{lem:involutions} that the additional condition $php=2p$ used in \cite{B1} is satisfied if and only if $p=p\circ *$, or equivalently, if $p$ factors through the cofixed point set of the involution:
\[
p: M_{ee} \sra (M_{ee})_{\z} = H_0(\z;M_{ee}) \to M_e
\]
It follows that the notion in \cite[p.246]{R} is equivalent to that in \cite{B1}, except that $M_{ee}$ is assumed to be the ground ring $R$ in the former. In that case, our involution is simply $r^* = \epsilon \bar r$, where $\epsilon=\pm 1$ and $r\mapsto \bar r$ is the given involution on the ring $R$. 

Then $\epsilon$-symmetric forms in the sense of Ranicki become hermitian forms in the sense defined below. In particular, Ranicki's $(+1)$-symmetric forms are different from the notion of {\em symmetric} form in this paper: We reserve it for the easiest case where both involutions, $*$ and $\dagger$, are trivial.

\subsection{Non-commutative quadratic forms}
\begin{defn}\label{def:quadratic} A {\em (non-commutative) quadratic form} on an abelian
group $A$ with values in a (non-commutative) quadratic group 
$\frak M= (M_{e} \overset{h}{\to} M_{ee} \overset{p}{\to} M_{e})$ is given by a bilinear map $\lambda:A\times A \to M_{ee}$ and a map $\mu:A\to M_{e}$ satisfying
\begin{enumerate}
\item $\mu(a+a')= \mu(a)+\mu(a')+p \circ \lambda(a,a')$ \ and
\item $h  \circ \mu(a) = \lambda(a,a) \ \forall a,a' \in A$.
\end{enumerate}
We say that $\mu$ is a {\em quadratic refinement} of $\lambda$: Property (i) says that $\mu$ is quadratic and property (ii) means that it ``refines'' $\lambda$. The notation  $M_e$ and $M_{ee}$ was designed (by Baues) to reflect the number of variables ({\bf e}ntries) of the maps $\mu$ and $\lambda$ respectively. He also writes $\lambda=\lambda_{ee}$ and $\mu=\lambda_e$, however, we decided not to follow that part of the notation.

We write $(\lambda,\mu):A\to \frak M$ for such quadratic forms and we always assume that the quadratic group $\frak M$ is part of the data for $(\lambda,\mu)$. 
This means that the morphisms in the category of quadratic forms are pairs of morphisms 
\[
\alpha:A\to A'  \quad \text{ and } \quad \beta=(\beta_e,\beta_{ee}): \frak M \to \frak{M}'
\]
such that both diagrams involving $\lambda, \lambda', \mu, \mu'$ commute.
\end{defn}
\begin{lem}\label{lem:symmetries}
Let $(\lambda,\mu):A\to \frak M$ be a quadratic form as above. Then
$\lambda$ is {\em hermitian} with respect to the involution $*=hp-\id$ on $M_{ee}$:
\[
\lambda(a',a) = \lambda(a,a')^*
\]
and 
 $\mu$ is {\em hermitian} with respect to the anti-involution $\dagger =ph -\id$ on $M_e$:
\[
\mu(-a) =  \mu(a)^\dagger
\]
\end{lem}
\begin{proof}
As a consequence of conditions (i) and (ii) we get
\begin{align*}
& \lambda(a,a)+ \lambda(a',a') + \lambda(a',a) + \lambda(a,a') = \lambda(a+a', a+a') =\\
& h\circ\mu(a+a') = h(\mu(a)+ \mu(a') +p\circ \lambda(a,a'))=\\
& \lambda(a,a)+ \lambda(a',a') + hp(\lambda(a,a'))
\end{align*}
or equivalently, $\lambda(a',a) = (hp-\id)\lambda(a,a') = \lambda(a,a')^*$. 
Similarly,
\begin{align*}
& 0=\mu(0) = \mu(a -a) = \mu(a) + \mu(-a) + p\circ \lambda(a,-a) =\\
& \mu(a) + \mu(-a) - p\circ h\circ \mu(a) =\mu(-a) + (\id - ph)\mu(a)
\end{align*}
or equivalently, $\mu(-a) = \dagger\circ \mu(a) =:\mu(a)^\dagger$.
\end{proof}
Starting with a hermitian form $\lambda$ with values in a group with involution $(M,*)$, the first step in finding a quadratic refinement for $\lambda$ is to find a quadratic refinement $\frak M$ of $(M,*)$ in the sense of Definition~\ref{def:quadratic group refinement}, motivating our terminology.

\subsection{Universal quadratic refinements}

\begin{ex}\label{ex:universal}
Given a hermitian form $\lambda:A \times A \to (M,*)$, one gets a quadratic refinement $\mu_\lambda$ of $\lambda$ as follows. Set $M_{ee}:=M$ and define the universal target $M_{e}:=M_{ee} \times_\lambda A$ to be the group consisting of pairs $(m,a)$ with $m\in M_{ee}$ and $a\in A$ and multiplication given by
\[
(m,a) + (m',a') := (m + m' - \lambda(a,a'), a + a')
\]
In other words, $M_{e}$ is the central extension
\[
\xymatrix{
1 \ar[r]  & M_{ee} \ar[r]& M_{ee} \times_\lambda A  \ar[r]&  A \ar[r] & 1
}
\]
determined by the cocycle $\lambda$, compare Section~\ref{sec:presentations}. 
It follows that $M_e$ is commutative if and only if $\lambda$ is {\em symmetric} in the naive sense that $\lambda(a',a) = \lambda(a,a')$. Set
\[
 \quad p_\lambda(m):=(m,0),\quad  h_\lambda(m,a) :=  m +m^* + \lambda(a,a)
 \]
We claim that $\frak M_\lambda := (M_{ee}\overset{p_\lambda}{\to} M_e\overset{h_\lambda}{\to}M_{ee})$ is a quadratic group as in Definition~\ref{def:quadratic group}. It is clear that $p_\lambda$ is a homomorphism with image in the center of $M_e$. The homomorphism property of $h_\lambda$ follows from the fact that $\lambda$ is bilinear and hermitian:
\begin{align*}
& h_\lambda((m,a)+(m',a'))=
 h_\lambda(m+m'-\lambda(a,a'), a+ a') =\\
& (m+m'-\lambda(a,a')) + (m+m'-\lambda(a,a'))^* + \lambda(a+a',a+a')=\\
&(m+m^* + \lambda(a,a)) + (m'+ m'^* +\lambda(a',a'))=h_\lambda(m,a)+ h_\lambda(m',a')
\end{align*}
Condition (iii) of a quadratic group is also checked easily:
\begin{align*}
& h_\lambda p_\lambda h_\lambda(m,a) = h_\lambda(m+m^* + \lambda(a,a),0) =\\
& (m+m^*+\lambda(a,a)) + (m+m^*+\lambda(a,a))^* = \\
& 2(m+m^*+\lambda(a,a))= 2h_\lambda(m,a)
\end{align*}
We also see that 
\[
(h_\lambda p_\lambda -\id)(m) = h_\lambda(m,0) - m = (m+m^*) - m = m^*
\]
which means that $\frak M_\lambda$ ``refines'' (in the sense of Definition~\ref{def:quadratic group refinement}) the group with involution $(M,*)$. 
Finally,  setting $\mu_\lambda(a):= (0,a)$,  we claim that $(\lambda,\mu_\lambda):A\to \frak M_\lambda$ is a quadratic refinement of $\lambda$. We need to check properties (i) and (ii) of a quadratic form (Definition~\ref{def:quadratic}):
(i) is simply $h_\lambda \circ \mu_\lambda(a) = h_\lambda(0,a) = \lambda(a,a)$,
and (ii) explains why we used a sign in front of $\lambda$ in our central extension:
\begin{align*}
& \mu_\lambda(a)+\mu_\lambda(a') + p_\lambda \circ\lambda(a,a') =
(0,a) + (0,a') + (\lambda(a,a'),0) =\\
& (-\lambda(a,a'),a+a') + (\lambda(a,a'), 0) = (0,a+a') = \mu_\lambda(a+a')
\end{align*}
\end{ex}

The following result will show that $\mu_\lambda$ is indeed a {\em universal} quadratic refinement of $\lambda$. This is the content of the first statement in the theorem below. It follows from the second statement because for any quadratic refinement $\mu$ of $\lambda$ it shows that forgetting the quadratic data gives canonical isomorphisms
\[
\QF(L\circ R(\lambda,\mu), (\lambda,\mu)) \cong \HF( R(\lambda,\mu) , R(\lambda,\mu)) = \HF(\lambda,\lambda)
\]
where $\QF$ respectively $\HF$ are (the morphisms in) the categories of quadratic respectively hermitian forms. Since 
\[
L\circ R(\lambda,\mu) = L(\lambda)= (\lambda,\mu_\lambda)
\]
 and the morphisms in the category $ \QR_\lambda$ of quadratic refinements of $\lambda$ 
 by definition all lie over the identity of $\lambda$, the set $ \QR_\lambda(\mu_\lambda,\mu)$
contains a unique element, namely the required universal morphism $\mu_\lambda \to \mu$.

\begin{thm}\label{thm:universal}
The quadratic form $(\lambda,\mu_\lambda)$ is initial in the category of quadratic refinements of $\lambda$. In fact, the forgetful functor $R(\lambda,\mu)=\lambda$ from the category of quadratic forms to the category of hermitian forms has a left adjoint $L:\HF\to \QF$ given by $L(\lambda):= (\lambda,\mu_\lambda)$.
\end{thm}

\begin{proof}
We have to construct natural isomorphisms
\[
\QF((\lambda,\mu_{\lambda}), (\lambda',\mu')) = \QF(L(\lambda), (\lambda',\mu')) \cong \HF( \lambda , R(\lambda',\mu')) = \HF(\lambda, \lambda')
\]
for any quadratic form $(\lambda',\mu')$ and hermitian form $\lambda$. Recall that the morphisms in $\QF$ are pairs $\alpha: A\to A'$ and $\beta=(\beta_e,\beta_{ee}): \frak M \to \frak M'$ such that the relevant diagrams commute. This implies that forgetting about the quadratic datum $\beta_e$ gives a natural map from the left to the right hand side above. 

Given a morphism $(\alpha,\beta_{ee}): \lambda\to \lambda'$ consisting of homomorphisms $\alpha:A\to A'$ and $\beta_{ee}: (M_{ee},*) \to (M_{ee}',*')$ such that
\[
\lambda'(\alpha(a_1), \alpha(a_2)) = \beta_{ee} \circ \lambda(a_1, a_2) \in M_{ee}' \quad \forall\ a_i\in A
\]
we need to show that there is a {\em unique} homomorphism $\beta_e: M_e \to M_e'$ such that the following 3 diagrams commute:
\[
\xymatrix{
\ar @{} [dr] |{(1)} M_{e} \ar[r]^{h} \ar[d]_{\beta_e} & M_{ee} \ar[d]^{\beta_{ee}} 
& \ar @{} [dr] |{(2)} M_{ee} \ar[r]^{p} \ar[d]_{\beta_{ee}} & M_{e} \ar[d]^{\beta_{e}} 
& \ar @{} [dr] |{(3)} A \ar[r]^{\mu_{\lambda}} \ar[d]_{\alpha} & M_{e} \ar[d]^{\beta_{e}} \\
 M_{e}' \ar[r]^{h'} & M_{ee}'
 &  M_{ee}' \ar[r]^{p'} & M_{e}'  
 &  A' \ar[r]^{\mu'} & M_{e}'  
}
\]
We will now make use of the fact that $M_e=M_{ee} \times_\lambda A$ because $\mu_\lambda$ is given as in Example~\ref{ex:universal}. In this case, diagrams (2) and (3) are equivalent to
\[
\beta_e(m,0) = p'\circ\beta_{ee}(m)  \quad \text{ and } \quad \beta_e(0,a) = \mu'\circ \alpha(a)
\]
because $p(m) = (m,0)$ and $\mu_\lambda(a) = (0,a)$. 
This implies directly the uniqueness of $\beta_e$. For existence, we only have to check that the formula
\[
\beta_e(m,a) := p'\circ\beta_{ee}(m) + \mu'\circ \alpha(a)
\]
gives indeed a group homomorphism $M_e \to M_e'$ that makes diagram (1) commute. Note that the image of $p'$ is central in $M_e'$ and hence the order of the summands does not matter. We have
\begin{align*}
& \beta_e((m,a)+(m',a'))=
 \beta_e(m+m'-\lambda(a,a'), a+ a') &=\\
& p'\circ\beta_{ee}(m+m'-\lambda(a,a')) + \mu'\circ \alpha(a+a')&=\\
&p'\circ\beta_{ee}(m) + p'\circ\beta_{ee}(m') - p'\circ \lambda'(\alpha(a),\alpha(a')) + \mu'\circ \alpha(a+a')&=\\
&p'\circ\beta_{ee}(m)  + p'\circ\beta_{ee}(m') + \mu'\circ \alpha(a) + \mu'\circ \alpha(a')&=\\
&\beta_e(m,a)+ \beta_e(m',a')
\end{align*}
To get to the fourth line,  we used property (ii) of a quadratic form to cancel the term $p'\circ \lambda'(\alpha(a),\alpha(a'))$. For the commutativity of diagram (1) we use property (i) of a quadratic form, as well as the fact that $\beta_{ee}$ preserves the involution $*$:
\begin{align*}
&h'\circ \beta_e(m,a) = h' (p'\circ \beta_{ee}(m) + \mu'\circ\alpha(a))=\\
&h'p'(\beta_{ee}(m)) + \lambda'(\alpha(a),\alpha(a)) = \beta_{ee}(m)^{*'} + \beta_{ee}(m) + \beta_{ee} \circ\lambda(a,a) =\\
 & \beta_{ee}(m^*+ m +\lambda(a,a)) = \beta_{ee}\circ h(m,a) 
\end{align*}
This finishes the proof of left adjointness of $L:\HF\to \QF$.
\end{proof}
If the bilinear form $\lambda$ happens to be  {\em symmetric}, or more precisely, if it takes values in a group $M_{ee}$ with {\em trivial} involution $*$, then the above construction still gives a quadratic refinement $\mu_\lambda$. Its target quadratic group $\frak M_\lambda$ has the properties that $M_e$ is abelian and $h_\lambda p_\lambda = 2\id$. 
It is not hard to see that our construction above leads to the following result.

\begin{thm}\label{thm:universal}
For any symmetric form $\lambda$ one can functorially construct a quadratic form $(\lambda,\mu_\lambda)$ that is initial in the category of quadratic refinements of $\lambda$ with trivial involution $*$. In fact, the forgetful functor $R(\lambda,\mu)=\lambda$ from the category of quadratic forms with trivial involution $*$ to the category of symmetric forms has a left adjoint $L(\lambda)= (\lambda,\mu_\lambda)$.
\end{thm}

\begin{rem}
It follows from the above considerations that a quadratic form $(\lambda,\mu)$ is universal if and only if the homomorphism
\[
M_{ee} \times_\lambda A \to M_e  \quad \text{ given by } \quad (m,a)\mapsto p(m) + \mu(a)
\]
is an isomorphism. This is turn is equivalent to
\begin{enumerate}
\item $p:M_{ee} \to M_e$ is injective and
\item $\mu : A \to M_e/\im(p)$ is an isomorphism.
\end{enumerate}
\end{rem}

\subsection{Commutative quadratic groups and forms}
The case where $*$ is non-trivial but the anti-involution $\dagger$ on $M_e$ is trivial is even more interesting. In this case, $\lambda$ is still hermitian with respect to $*$ but one is only interested in quadratic refinements $\mu$ that are symmetric in the sense that $\mu(-a)=\mu(a)$. This case deserves its own definition:

\begin{defn}\label{def:commutative quadratic group} A {\em commutative quadratic group}
\[
\frak M = (M_{e} \overset{h}{\to} M_{ee} \overset{p}{\to} M_{e})
\]
consists of two abelian groups $M_e, M_{ee}$ and two homomorphism $h,p$ satisfying $ph=2\id$. 
\end{defn}

In fact, a commutative quadratic group is the same thing as a non-commutative quadratic group with trivial anti-involution $\dagger$. This comes from the fact that the squaring map $x\mapsto 2x$ is a homomorphism if and only if $M_e$ is commutative. Our universal example $\frak M_\lambda$ is in general {\em not} commutative because one gets in this case
\begin{align*}
& \dagger_\lambda(m,a) = p_\lambda \circ h_\lambda(m,a) - (m,a) = p_\lambda(m+m^* +\lambda(a,a)) - (m,a) =\\
& (m+m^* +\lambda(a,a),0) + (-m-\lambda(a,a),-a) = (m^*, -a)
\end{align*}
However, we shall see in Theorem~\ref{thm:commutative universal} that we can just divide by these relations $(m,a) = (m^*,-a)$ to obtain another universal quadratic refinement of a given hermitian form $\lambda$ but this time with values in a {\em commutative} quadratic group. Before we work this out, let us mention the essential example from topology.
\begin{ex}
Consider a manifold $X$ of dimension~$2n$ and let $\frak M$ be as in (M) from Example~\ref{ex:M} with $M=\Z[\pi_1X]$.  In particular, we have $ph-\id=\dagger=\id$ but in general the involution $*$ is non-trivial. On group elements, it is given by
\[
g^* := (-1)^n w_1(g) g^{-1}
\]
with $w_1$ (induced by) the first Stiefel-Whitney class of $X$. Then the equivariant intersection form $\lambda=\lambda_X$ on $\pi_nX$ is bilinear and hermitian as required. Moreover, the self-intersection invariant $\mu_X$ defined by Wall \cite{Wa} gives a quadratic refinement of $\lambda_X$, at least on the subgroup $A$ of elements represented by immersed $n$-spheres with vanishing normal Euler number. 
\end{ex}

In our main Theorem~\ref{thm:commutative universal} below, we shall use the following
\begin{lem}\label{lem:square}
If $(\lambda,\mu):A\to \frak M$ is a commutative quadratic form, then $\mu(n\cdot a)=n^2\cdot \mu(a)$ for all integers $n\in \Z$.
\end{lem}
Here we say that a quadratic form $(\lambda,\mu):A\to \frak M$ is {\em commutative} if the target quadratic group $\frak M$ is commutative, i.e.\ if the anti-involution $\dagger$ is trivial (Definition~\ref{def:quadratic}). 
\begin{proof}
Since the involution $\dagger=ph-\id$ is trivial by assumption, we already know that $\mu(-a)=\mu(a)$ from Lemma~\ref{lem:symmetries}. Thus it suffices to prove the claim for positive $n>1$ by induction:
\begin{align*}
\mu((n+1)\cdot a) &= \mu(n\cdot a) + \mu(a) + p\circ\lambda(n\cdot a,a) \\
&= n^2\cdot \mu(a) + \mu(a) + n\cdot p\circ h\circ \mu(a) \\
&= (n^2+1)\cdot \mu(a) + n \cdot 2\cdot \mu(a) = (n+1)^2\cdot \mu(a)
\end{align*}
Here we again used the fact that $p\circ h=2\id$.
\end{proof}

\begin{thm}\label{thm:commutative universal}
Any hermitian bilinear form $\lambda$ has a universal commutative quadratic refinement. In fact, the forgetful functor $R(\lambda,\mu)=\lambda$ from the category $\CQF$ of commutative quadratic forms to the category $\HF$ of hermitian forms has a left adjoint $L:\HF\to\CQF, L(\lambda)=(\lambda,\mu^c_\lambda)$.
\end{thm}

\begin{proof}
As hinted to above, we will force the anti-involution $\dagger$ to be trivial in the universal construction of Theorem~\ref{thm:universal}. This means that we should define the universal (commutative) group $M^c_e$ as the quotient of our previously used group $M_{ee} \times_\lambda A$ by the relations
\begin{align*}
0 &= (m^*, -a) - (m,a) = (m^*,-a) + (-m-\lambda(a,a)),-a) \\
& =(m^* - m - 2\lambda(a,a), -2a)
\end{align*}
By setting $a$ respectively $m$ to zero, these relations imply
\[
(m^*,0) = (m,0)  \quad \text{ and } \quad (-2\lambda(a,a), -2a) = 0
\]
Vice versa, these two types of equations imply the general ones and hence we see that $M^c_e$ is the quotient of the centrally extended group
\[
\xymatrix{
1 \ar[r]  & M_{ee}/(m^*=m) \ar[r]& M_{ee}/(m^*=m) \times_\lambda A  \ar[r]&  A \ar[r] & 1
}
\]
by the relations $(-2\lambda(a,a), -2a) = 0$. We write elements in $M^c_e$ as $[m,a]$ with the above relations understood. It then follows that $p^c_\lambda(m):= [m,0]$ is a homomorphism $M_{ee}\to M^c_e$ (which is in general not any more injective). Moreover, our original formula leads to a homomorphism $h^c_\lambda:M^c_e\to M_{ee}$ given by
\[
h^c_\lambda[m,a]:= h_\lambda(m,a)= m+m^* + \lambda(a,a)
\]
To see that this is well defined, observe $h_\lambda(m^*,0) = m+m^* = h_\lambda(m,0)$ and
\[
h_\lambda(-2\lambda(a,a),-2a)= -4\lambda(a,a) + \lambda(-2a,-2a) =0 
\]
Finally, we set $\mu^c_\lambda(a):=[0,a]$ to obtain a commutative quadratic refinement of $\lambda$ which is proven exactly as in Theorem~\ref{thm:universal}. 

To show that $\mu^c_\lambda$ is universal, or more generally, that $L(\lambda):=(\mu^c_\lambda,\lambda)$ is a left adjoint of the forgetful functor $R$, we proceed as in the proof of Theorem~\ref{thm:universal}: 
We are given a morphism $(\alpha,\beta_{ee}): \lambda\to \lambda'$ consisting of homomorphisms $\alpha:A\to A'$ and $\beta_{ee}: (M_{ee},*) \to (M_{ee}',*')$ such that
\[
\lambda'(\alpha(a_1), \alpha(a_2)) = \beta_{ee} \circ \lambda(a_1, a_2) \in M_{ee}' \quad \forall\ a_i\in A.
\]
We need to show that there is a {\em unique} homomorphism $\beta_e: M^c_e \to M_e'$ such that the three diagrams from the proof of Theorem~\ref{thm:universal} commute.
We can use the same formulas as before, if we check that they vanish on our new relations in $M^c_e$. For this we'll have to use that the given quadratic group $\frak M'$ is {\em commutative}. Recall the formula
\[
\beta_e(m,a) = p'\circ\beta_{ee}(m) + \mu'\circ \alpha(a)
\]
Splitting our relations into two parts as above, it suffices to show that
\[
p'\circ\beta_{ee}(m^*) = p'\circ\beta_{ee}(m)  \quad \text{ and } \quad \beta_e(-2\lambda(a,a),-2a)=0
\]
The first equation follows from part (iii) of Lemma~\ref{lem:involutions} and the fact that we are assuming that $\dagger'=\id$:
\begin{align*}
 p'\circ\beta_{ee}(m^*) = (p'\circ *')(\beta_{ee}(m)) =  (\dagger' \circ p')(\beta_{ee}(m)) = p'\circ\beta_{ee}(m)
\end{align*}
For the second equation we compute:
\begin{align*}
& \beta_e(-2\lambda(a,a),-2a)=p'\circ\beta_{ee}(-2\lambda(a,a)) + \mu'\circ \alpha(-2a)=\\
& -2(p' \circ\lambda'(\alpha(a),\alpha(a))) + \mu'\circ \alpha(-2a)=\\
 &-2(\mu'(\alpha(a) + \alpha(a)) - \mu'(\alpha(a)) - \mu'(\alpha(a))) + \mu'(-2\alpha(a))=\\
 &-2(4\mu'(\alpha(a)) - 2\mu'(\alpha(a))) + 4\mu'(\alpha(a))= -4\mu'(\alpha(a)) + 4\mu'(\alpha(a))=0
\end{align*}
We used Lemma~\ref{lem:square} for $n=\pm 2$ and hence the commutativity of $\frak M$.
\end{proof}

\subsection{Symmetric quadratic groups and forms}

The simplest case of a quadratic group is where both $*$ and $\dagger$ are trivial. Let's call such a quadratic group $\frak M= (M_{e} \overset{h}{\to} M_{ee} \overset{p}{\to} M_{e})$ {\em symmetric}. Equivalently, this means that $hp=2\id=ph$ (and hence $M_e$ is commutative). Then a quadratic form $(\lambda,\mu):A\to\frak M$ will automatically be {\em symmetric} in the sense that
\[
\lambda(a,a') = \lambda(a',a)  \quad \text{ and } \quad \mu(-a) = \mu(a)\quad \forall \ a\in A.
\]
We call $\mu$ a {\em symmetric quadratic refinement} of $\lambda$ and obtain a category of symmetric quadratic forms with a forgetful functor $R$ to the category of symmetric forms. It is not hard to show that the construction in Theorem~\ref{thm:commutative universal} gives a universal symmetric quadratic refinement $\mu^c_\lambda$ for any given symmetric bilinear form $\lambda$. More precisely, 

\begin{thm}\label{thm:symmetric universal}
Any symmetric bilinear form $\lambda$ has a universal symmetric quadratic refinement. In fact, the forgetful functor $R(\lambda,\mu)=\lambda$ from the category $\SQF$ of symmetric quadratic forms to the category $\SF$ of symmetric forms has a left adjoint $L:\HF\to\CQF, L(\lambda)=(\lambda,\mu^c_\lambda)$.
\end{thm}

\begin{rem}\label{rem:injective}
We observe that the map $p^c_\lambda: M_{ee}\to M^c_e$ is a monomorphism in this easiest, symmetric, case, just like it was in the hardest, non-commutative, case. This can be seen by noting that the first set of relations $(m^*,0)=(m,0)$ is redundant if the involution $*$ is trivial. Therefore, if $0=p^c_\lambda(m) = [m,0]$ then $(m,0)$ must come from the second set of relations, i.e.\ it must be of the form 
\[
(m,0)=(-2 \lambda(a,a), -2a)  \quad \text{ for some } a\in A.
\]
 This implies that $2a=0$ and hence $\lambda(2a,a)=0$ which in turn means $m=0$.
\end{rem}

\begin{cor}\label{cor:exact} There is an exact sequence
\[
\xymatrix{
1 \ar[r]  & M_{ee}\ar[r]^p & M_{e}^c  \ar[r]&  \Z_2 \otimes A \ar[r] & 1
}
\]
\end{cor}

\begin{exs}\label{ex:quadratic}
If $M_{ee}=M_{e}$ then $h=\id$ and $p=2\id$ is a canonical choice for which $\mu$ is determined by $\lambda$.  Another canonical choice is $p=\id$ and $h =2\id$. Then a quadratic refinement of $(M_{e}, h, p)$ with this choice exists exactly for even forms, at least for free groups $A$. Moreover, if $M_{ee}$ has no 2-torsion then a quadratic refinement is uniquely determined by the given even form. 

At the other extreme, consider $M_{ee}=M_{e}=\z$. If $A$ is a finite dimensional $\z$-vectorspace then non-singular symmetric bilinear forms $\lambda$ are classified by their rank and their {\em parity}, i.e.~whether they are even or odd, or equivalently, whether they admit a quadratic refinement or not. In the even case, quadratic forms $(\lambda,\mu)$ are classified by rank and {\em Arf invariant}. This additional invariant takes values in $\z$ and vanishes if and only if $\mu$ takes more elements to zero than to one (thus the Arf invariant is sometimes referred to as the ``democratic invariant''). 

If $\lambda$ is odd then the following trick allows one to still define Arf invariants and it motivates the introduction of $M_{e}$. Let again $A$ be a finite dimensional $\z$-vectorspace, $M_{ee}=\z$ and $M_{e}=\Z_4$ with the unique nontrivial homomorphisms $h, p$. Then any non-singular symmetric bilinear form $\lambda$ has a quadratic refinement $\mu$ and quadratic forms $(\lambda,\mu)$ are classified by rank and an Arf invariant with values in $\Z_8$. If $\lambda$ is even, this agrees with the previous Arf invariant via the linear inclusion $\z \subset \Z_8$.
\end{exs}

\subsection{Presentations for universal quadratic groups} \label{sec:presentations}
Consider a central group extension
\[
1 \to M \to G \overset{\pi}{\to} A \to 1
\]
and assume that $M$ and $A$ have presentations $ \langle m_i | n_j \rangle$ respectively $ \langle a_k | b_\ell \rangle $. To avoid confusion, we write groups multiplicatively for a while and switch back to additive notation when returning to hermitian forms. 

It is well known how to get a presentation for $G$: Pick a section $s:A\to G$ with $s(1)=1$ which is not necessarily multiplicative. Write a relation in $A$ as $b_\ell = a'_1 \cdots a'_r$, where $a'_i$ are generators of $A$ or their inverses, then
\[
1=s(1)=s(b_\ell)=s(a'_1)\cdots s(a'_r) \, w_\ell
\]
where $w_\ell=w_\ell(m_i)$ is a word in the generators of $M$. This equation follows from the fact that the projection $\pi$ is a homomorphism and for simplicity we have identified $M$ with its image in $G$. We obtain the presentation
\[
G = \langle m_i, \alpha_k\, |\, n_j, [m_i,\alpha_k], \beta_\ell \, w_\ell \rangle 
\]
where $\alpha_k:=s(a_k)$ and $\beta_\ell:=s(a'_1)\cdots s(a'_r)$ is the same word in the $\alpha_k$ as $b_\ell$ is in the $a_k$. The commutators $[m_j,\alpha_k]$ arise because we are assuming that the extension is central, in a more general case one would write out the action of $A$ on $M$. 

It will be useful to rewrite this presentation as follows. Observe that the section $s$ satisfies
\[
s(a_1a_2)=s(a_1)s(a_2) c(a_1,a_2)
\]
for a uniquely determined {\em cocycle} $c: A \times A\to M$. By induction one shows that
\begin{align*}
& s(a_1\cdots a_r) = s(a_1\cdots a_{r-1})s(a_r) c(a_1\cdots a_{r-1},a_r)= \cdots = \\
&s(a_1)\cdots s(a_r) c(a_1, a_2) c(a_1a_2,a_3)c(a_1a_2a_3,a_4)\cdots c(a_1\cdots a_{r-1},a_r)
\end{align*}
Comparing this expression with the definition of the word $w_\ell$ in the presentation of $G$, it follows that 
\[
w_\ell= c(a'_1, a'_2) c(a'_1a'_2,a'_3)\cdots c(a'_1\cdots a'_{r-1},a'_r) \ \in M
\]
so that the above presentation of $G$ is entirely expressed in terms of the cocycle $c$ (and does not depend on the section $s$ any more). 

Now assume that $\lambda:A \times A\to M$ is a hermitian form with respect to an involution $*$ on $M$. Then the universal (non-commutative) quadratic group $M_e$ from Example~\ref{ex:universal} is a central extension as above with cocycle $c=\lambda$. Reverting to additive notation, we see that 
\begin{align*}
w_\ell&= \lambda(a'_1, a'_2) + \lambda(a'_1+a'_2,a'_3)+ \cdots + \lambda(a'_1+\cdots +a'_{r-1},a'_r)\\
&= \sum_{1\leq i < j \leq r} \lambda(a'_i,a'_j)
\end{align*}
where the ordering of the summands is irrelevant because $M$ is central in $M_e$. Summarizing the above discussion, we get.
\begin{lem} \label{lem:presentation}
The universal (non-commutative) quadratic group $M_e$ corresponding to the hermitian form $\lambda$ has a presentation
\[
M_e= \langle m_i, \alpha_k\, |\, n_j, [m_i,\alpha_k], \beta_\ell + \sum_{1\leq i < j \leq r} \lambda(a'_i,a'_j) \rangle 
\]
where the generators $m_i, \alpha_k$ and words $n_j,\beta_\ell$ are defined as above. Moreover, the universal quadratic refinement $\mu: A\to M_e$ is a (in general non-multiplicative) section of the central extension and hence $\alpha_k = \mu(a_k)$ for the generators $a_k$ of $A$.
\end{lem}
As discussed in Theorem~\ref{thm:commutative universal}, we get the universal {\em commutative} quadratic group $M_e^c$ for $\lambda$ by adding the relations $(m^*,0) = (m,0)$ and $(-2\lambda(a,a),-2a)=0$. The latter can be rewritten in the form $2(0,a)=(\lambda(a,a),0)$. In the current notation, where $(m,0)$ is identified with $m\in M$, we obtain the relations
\[
m^* = m  \quad \text{and} \quad 2 \mu(a)= \lambda(a,a)  \ \in M_e^c \quad\forall\ m\in M, a\in A.
\]
Recalling that $A,M$ and $M_e^c$ are {\em commutative} groups, we can write our presentation in that category
to obtain
\begin{lem} \label{lem:commutative presentation}
The universal (commutative) quadratic group $M^c_e$ corresponding to the hermitian form $\lambda:A \times A\to M$ has a presentation
\[
M^c_e= \langle m_i, \mu(a_k)\, |\, n_j, \beta_\ell + \sum_{1\leq i < j \leq r} \lambda(a'_i,a'_j) , m^* = m  ,2 \cdot\mu(a)= \lambda(a,a) \ \rangle 
\]
Here $ \langle m_i | n_j \rangle$ is a presentation of $M$ and $a_k$ are generators of $A$. Moreover, for every relation $b_\ell = \sum_{i=1}^r a'_i$ in $A$, we use the word $\beta_\ell:= \sum_{i=1}^r \mu(a'_i)$.
\end{lem}

\subsection{Twisted intersection invariants and a universal quadratic group}

If we apply this construction to the universal inner product on order $n$ rooted trees
\[
\langle \ , \ \rangle:\ \cL_{n}(m) \times \cL_{n}(m) \longrightarrow \cT_{2n}(m)=:\cT_{2n}(m)_{ee}
\]
we obtain a universal symmetric quadratic refinement 
\[
q:=\mu^c_{\langle \ , \ \rangle}: \cL_{n}(m)\to \cT_{2n}(m)^c_e
\]
Let us compute the presentation from Lemma~\ref{lem:commutative presentation} in this case.
Recall that the generators of $\cL_{n}(m)$ are rooted trees $J$ of order $n$ and the relations are the AS and IHX 
relations from Figure~\ref{fig:relations}.  Similarly, $\cT_{2n}(m)$ is generated by unrooted trees $t$ of order $2n$, modulo the same relations. Putting these together, we see that $\cT_{2n}(m)^c_e$ is generated by unrooted trees $t$ of order $2n$ and elements $q(J)$, one for each rooted tree $J$ of order $n$. The three types of relations from Lemma~\ref{lem:commutative presentation} are:
\begin{itemize}
\item[$n_j:$] Relations in $M=\cT_{2n}(m)$ are ordinary AS and IHX relations for unrooted trees $t$,
\item[$\beta_\ell:$]  Every relation $b_\ell$ in $A=\cL_n(m)$ is an AS-relation $J+\bar J=0$ or an IHX-relation $I-H+X=0$. We obtain the following 
 {\em twisted}  AS- respectively IHX-relations:
 \begin{align*}
&0 = q(J) + q(\bar J) + \langle J, \bar J \rangle \\
&0= q(I)+q(H)+q(X) - \langle I,H \rangle + \langle I,X \rangle - \langle H,X \rangle
\end{align*}
\item[$c:$] $2\cdot q(J) = \langle J, J \rangle $
\end{itemize}
The last relation $c$ builds in the commutativity of the universal group as discussed above because we are in the easiest, symmetric, setting where the involution $*$ is trivial. Using relation $c$, the twisted AS relation simply becomes
\[
q(\bar J) = q(-J) = q(J)
\]
which was expected since we are in the symmetric case. This relation means that the orientation of $J$ is irrelevant when forming $q(J)$ and in fact, with some care one can see that the twisted IHX-relation makes sense for unoriented trees.

\begin{lem}\label{lem:quadratic-form}
This is a presentation for the target group $ \cT^\iinfty_{2n}(m)$ of twisted Whitney towers from Definition~\ref{def:twisted-tree-groups}. 
\end{lem}
\begin{proof}
The translation comes from setting $J^\iinfty=q(J)$ for rooted trees $J$ (and keeping unrooted trees unchanged). 
We need to show that the twisted IHX-relations in the original definition of $\mathcal T^\iinfty_{2n}(m)$ are equivalent to the twisted IHX-relations above, all other relations were already shown to agree.  This is very easy to see in the presence of the interior-twist relations: Together with the (untwisted) IHX-relations, they imply that
\[
0 = \langle I, I - H + X \rangle = \langle I, I \rangle - \langle I, H \rangle + \langle I, X \rangle = 2\cdot q(I) - \langle I, H \rangle + \langle I, X \rangle
\]
This last expression is exactly the difference between the two versions of the twisted IHX-relations.
\end{proof}

\begin{cor} \label{cor:quadratic} There is an isomorphism of symmetric quadratic groups
\[
\cT_{2n}(m)^c_e \cong  \cT^\iinfty_{2n}(m)
\]
which is the identity on $ \cT_{2n}(m)$ and takes $q(J)$ to $J^\iinfty$ for rooted trees $J$. 
The quadratic group structure on  $\cT^\iinfty_{2n}(m)$ is given by the homomorphisms $\cT_{2n}(m)\overset{p}{\to} \cT^\iinfty_{2n}(m)\overset{h}{\to}\cT_{2n}(m)$ which are uniquely characterized (for unrooted trees $t$ and rooted trees $J$) by
\[
p(t)=t  \quad \text{and} \quad h(t) = 2\cdot t, \ h(J^\iinfty) = \langle J,J \rangle 
\]
\end{cor}
Note that Theorem~\ref{thm:exact} is now a direct consequence of Corollary~\ref{cor:exact}.


\end{document}